\documentclass[12pt]{amsart}
\usepackage{amsmath, amsthm, amssymb, enumerate, color,amsxtra}

\usepackage[colorlinks=true, pdfstartview=FitV, linkcolor=blue, citecolor=blue, urlcolor=blue]{hyperref}

\makeatletter
\def\Ddots{\mathinner{\mkern1mu\raise\p@
\vbox{\kern7\p@\hbox{.}}\mkern2mu
\raise4\p@\hbox{.}\mkern2mu\raise7\p@\hbox{.}\mkern1mu}}
\makeatother

\def\Xint#1{\mathchoice
{\XXint\displaystyle\textstyle{#1}}%
{\XXint\textstyle\scriptstyle{#1}}%
{\XXint\scriptstyle\scriptscriptstyle{#1}}%
{\XXint\scriptscriptstyle\scriptscriptstyle{#1}}%
\!\int}
\def\XXint#1#2#3{{\setbox0=\hbox{$#1{#2#3}{\int}$}
\vcenter{\hbox{$#2#3$}}\kern-.5\wd0}}

\def\dashint{\Xint-}

\newcounter{mycounter}

\newtheorem{theorem}{Theorem}[section]

\newtheorem{corollary}[theorem]{Corollary}

\theoremstyle{definition}
\newtheorem{definition}[theorem]{Definition}
\newtheorem{example}[theorem]{Example}

\newtheorem{lemma}[theorem]{Lemma}

\newtheorem{question}[theorem]{Question}

\newtheorem{remark}[theorem]{Remark}

\def\loc{\text{loc}}

\def\al{{\alpha}}
\def\R{\mathbb R}

\def\T{\mathbb T}

\def\C{\mathbb C}
\def\W{\mathcal W}
\def\X{\mathcal X}
\def\ra{\rightarrow}
\def\bey{\begin{eqnarray*}}
\def\eey{\end{eqnarray*}}
 
 \allowdisplaybreaks

\def\D{{\mathbb D}}

\def\A{{\mathcal A}}
\def\M{{\mathcal M}}

\def\B{{\mathcal B}}

\begin{document}

\title{Unions of Lebesgue spaces and $A_1$ majorants}
\date{\today}

\author{Greg Knese}
\address{GK: Washington University in St. Louis, Department of
  Mathematics, St. Louis, MO}
\email{geknese@math.wustl.edu}

\author{John E. M$^\text{c}$Carthy}
\address{JM: Washington University in St. Louis, Department of
  Mathematics, St. Louis, MO}
\email{mccarthy@wustl.edu}

\author{Kabe Moen}
\address{KM: University of Alabama, Department of Mathematics, Tuscaloosa, AL}
\email{kmoen@as.ua.edu}

\thanks{GK partially supported by NSF grant DMS-1363239}
\thanks{JM partially supported by NSF grant DMS-1300280}
\thanks{KM partially supported by NSF grant DMS-1201504}

\subjclass[2010]{42B25, 42B35, 46E30, 30H10, 30H15}

\maketitle

\begin{abstract} We study two questions.  When does a function belong to the union of Lebesgue spaces and when does a function have an $A_1$ majorant?   We show these questions are fundamentally related.  For functions restricted to a fixed cube we prove that the following are equivalent: a function belongs to $L^p$ for some $p>1$; the function has an $A_1$ majorant; for any $p>1$ the function belongs to $L^p_w$ for some $A_p$ weight $w$.  We also examine the case of functions defined on $\R^n$ and give characterizations of the union of $L^p_w$ over $w\in A_p$ and when a function has an $A_1$ majorant on all of $\R^n$.

\end{abstract}

\section{Introduction and statement of the main results}
 While the $L^p$ spaces are considered
    fundamental spaces of interest in analysis, the weighted $L^p$
    spaces and the related study of $A_p$ weights are perhaps part of
    a more specialized area of analysis.  It is the goal of this
    article to show that the $L^p$ spaces considered in aggregate are
    intimately linked to these latter topics and to the notion of an
    $A_1$ majorant. 
    
We begin with the following question.  
 \begin{question} \label{ques:unionLp} When does a function belong to the union of $L^p$ spaces?
\end{question}
\noindent Question \ref{ques:unionLp} is vaguely stated on purpose.  By union,
we will either mean the union of $L^p$ as $p$ varies or the union of $L^p_w$ for a fixed $p$ and $w$ varying.  The union of $L^p$ spaces often arises when considering a general domain to define operators in harmonic analysis.  Several such operators are bounded on $L^p$ for all $1<p<\infty$, hence will take functions from $\bigcup_{p>1}L^p$ into itself.  
 
It turns out Question \ref{ques:unionLp} is closely related to the theory of weighted Lebesgue spaces and the action of the Hardy-Littlewood maximal operator on these spaces.  For our purposes, a weight is a positive locally integrable function.  An $A_1$ weight is one that satisfies 
$$Mw\leq Cw, \qquad a.e.$$
Here $M$ denotes the Hardy-Littlewood maximal operator 
$$Mf(x)=\sup_{Q\ni x}\frac{1}{|Q|}\int_Q |f|\,dx.$$ We exclude the
weight $w\equiv 0$ from belonging to $A_1$, and in this case we see
that if $w\in A_1$ then $w>0$ a.e.  The $A_1$ class of weights
characterizes when $M$ maps $L^1_w$ into $L^{1,\infty}_w$.  When
$1<p<\infty$, $M$ is bounded on $L^p_w$ exactly when $w\in A_p$:
$$ \Big(\frac{1}{|Q|}\int_Q w\,dx\Big)\Big(\frac{1}{|Q|}\int_Q
w^{-\frac{1}{p-1}}\,dx\Big)^{p-1}\leq C$$ for all cubes $Q$.  At the other
endpoint the $A_\infty$ class is defined to be the union of all $A_p$
for $p\geq 1$.  We now come to our second question.

\begin{question} \label{ques:A1maj} 
Given a measurable function, $f$, when does there exist an $A_1$
weight, $w$, such that
\begin{equation}\label{eqn:A1maj} |f|\leq w?\end{equation}
\end{question}

We call a weight satisfying \eqref{eqn:A1maj} an $A_1$ majorant
of $f$ and write $\M_{A_1}$ for the set of measurable functions
possessing an $A_1$ majorant.  As stated, Question \ref{ques:A1maj}
does not seem to have been considered before.  As far as we can tell,
the first notion of an $A_1$ majorant appeared in the 2012 article by
Rutsky \cite{Rut}.  In this paper a different definition of an $A_1$
majorant is given---one which requires the function and the weight to
{\it a priori} belong to a more restrictive class of functions.

 It turns out our problems split into two immediate cases: the local and global case.
Because of this we will use the notation $L^p$, $L^p_w$, $M_{A_1}$,
etc., when the domain does not matter and the notation $L^p(\Omega),
L^p_w(\Omega), \M_{A_1}(\Omega)$ for a fixed domain $\Omega$.  In
general, we will take $\Omega=Q$ where $Q$ is a cube in $\R^n$ (the
local case) or $\Omega=\R^n$ (the global case).  In the local case our
problem has a remarkably simple answer---one which reveals a close
connection between traditional $L^p$ spaces, weighted $L^p$ spaces,
and $A_1$ majorants.
 
\begin{theorem}\label{thm:localmain} 
Suppose $Q$ is a cube in $\R^n$ and $p_0$ is an exponent satisfying
$1<p_0<\infty$.  Then
$$\M_{A_1}(Q)=\bigcup_{p>1} L^p(Q)=\bigcup_{w\in
    A_{p_0}}L^{p_0}_w(Q).$$
\end{theorem}
The proof, as with most proofs in this article, is a
    synthesis of known important results; in this case the reverse
    H\"older inequality, a result of Coifman and Rochberg \cite{CR}
    (Theorem \ref{thm:coifroch} below), and basic properties of $A_p$
    weights. 
    The second equality in Theorem \ref{thm:localmain} reinforces the
    saying attributed to Antonio C\'ordoba ``there are no $L^p$ spaces
    only weighted $L^2$ spaces".

Theorem \ref{thm:localmain} has several extensions.  First, a function
may not have an $A_1$ majorant, but a power of it may.  Given $r>0$ we
define the class $\M_{A_1}^r$ to be the set of functions such that $|f|^r$
has an $A_1$ majorant (note: $\M_{A_1}=\M_{A_1}^1$).  We have the more
general result which implies Theorem \ref{thm:localmain} by taking $r=1$.

 \begin{theorem}\label{thm:localmainr} 
Let $Q$ be a cube in $\R^n$ and $r,p_0$ satisfy $0<r<p_0<\infty$.
Then
$$\M^r_{A_1}(Q)=\bigcup_{p>r} L^p(Q)=\bigcup_{w\in A_{\frac{p_0}{r}}}L^{p_0}_w(Q).$$
\end{theorem}

We further extend the theory to $A_\infty$ weights.

 \begin{theorem}\label{thm:localmaininfty} 
Let $Q$ be a cube in $\R^n$ and $p_0$, $0<p_0<\infty$.  Then
$$\bigcup_{r>0} \M^r_{A_1}(Q)=\bigcup_{p>0} L^p(Q)=\bigcup_{w\in A_{\infty}}L^{p_0}_w(Q).$$
\end{theorem}

One may inquire about $A_p$ majorants for $p>1$ or $A_\infty$
majorants, that is, given a function when does there exists $w\in A_p$, $1<p\leq \infty$, with $|f|\leq w$.  We denote such classes of
functions as $\M_{A_p}$ or $\M_{A_\infty}$.  Since the $A_p$ classes
are nested we have
$$\M_{A_1}\subset \M_{A_p}\subset \M_{A_q}\subset \M_{A_\infty}$$ for
$1\leq p\leq q\leq \infty$.  Once again in the local case we have the
following nice characterization.

 \begin{theorem}\label{thm:localApmajor} If $Q$ is a cube in $\R^n$ then
 $$\M_{A_1}(Q)=\M_{A_\infty}(Q).$$
\end{theorem}

As an application of the local theory we are able to extend these
results to Hardy spaces.  Note that we look at the
  ``complex analyst's Hardy space'' as opposed to the real analyst's
  Hardy space defined in terms of maximal functions.  Let $\D$ denote
  the unit disk in the plane and $\T$ denote its boundary.  Given $p$,
  $0<p<\infty$, let $H^p=H^p(\D)$ be the space of analytic functions
  ``normed'' by
$$\|f\|_{H^p}=\sup_{0<r<1}\left(\int_0^{2\pi}|f(re^{i\theta})|^p\frac{d\theta}{2\pi}\right)^{1/p}.$$
``Norm'' is in quotes since this is not a norm for $0<p<1$, but we
shall use norm notation $\|\cdot\|$ nonetheless.
The Nevanlinna class, denoted $N$, is the collection of analytic functions on $\D$ such that
$$\|f\|_{N}=\sup_{0<r<1}\int_0^{2\pi}\log^+|f(re^{i\theta})|\,\frac{d\theta}{2\pi}<\infty.$$
Functions in $N$ have nontangential limits almost everywhere on the
boundary so we may treat them as functions on the disk or the circle.
The Smirnov class $N^+$ consists of functions $f\in N$ such that
$$\lim_{r\ra
  1}\int_0^{2\pi}\log^+|f(re^{i\theta})|\,\frac{d\theta}{2\pi}
=\int_0^{2\pi}\log^+|f(e^{i\theta})|\,\frac{d\theta}{2\pi}.$$ 
It is
well known that
$$\bigcup_{p>0} H^p\subsetneqq N^+\subsetneqq N.$$
(see, e.g., the books by Duren \cite{Dur} or Rudin \cite{Rud}.)  The Smirnov class
is often considered a natural limit of $H^p$ as $p\to 0$.

A weight on the torus will be a positive function in $L^1(\T)$.  The
classes $A_1(\T)$, $A_p(\T)$, and $A_\infty(\T)$ are defined analogously
on $\T$.  The weighted Hardy space $H^p_w=H^p_w(\D)$ is the closure of
analytic polynomials in $L^p_w(\T)$.  Since we may identify the torus
$\T$ with $[0,2\pi]$ it is obvious that Theorems \ref{thm:localmainr}
and \ref{thm:localmaininfty} hold for $\T$.   While there
  are real variable definitions of weighted Hardy spaces, this
  classical definition has an intuitive appeal.

In \cite{McC}, while studying the range of Toeplitz operators, the second author showed that
$$N^+=\bigcup_{w\in \W}H^2_w$$
where $\W$ is the Szeg\H{o} class of weights satisfying
$$\int_{\T}\log w\,d\theta>-\infty.$$
We notice that if $w\in A_\infty(\T)$ then we have 
$$\Big(\int_\T w\,d\theta\Big)\exp\Big(-\int_{\T}\log
w\,d\theta\Big)<\infty,$$ in particular $A_\infty(\T)\subset \W$.
Using the above techniques we are able to give a characterization of
$\bigcup_{p>0}H^p$ in terms of weighted $H^p$ spaces.
\begin{theorem} \label{thm:Hp>0} Suppose $p_0$ is an exponent satisfying $0< p_0<\infty$.  Then
$$\bigcup_{p>0} H^p=\bigcup_{w\in A_\infty} H^{p_0}_w.$$
\end{theorem}
If we define  $H_{A_1}(\T)$ as functions in $N^+$ whose boundary
function is majorized by an $A_1(\T)$ weight, then we have the
following analog of Theorem \ref{thm:localmain} for Hardy spaces.
\begin{theorem} \label{thm:Hp>1} If $p_0$ is an exponent satisfying $1<p_0<\infty$, then
$$H_{A_1}(\T)=\bigcup_{p>1} H^p=\bigcup_{w\in A_{p_0}} H^{p_0}_w.$$
\end{theorem}

For functions on $\R^n$, the theory is not as nice.  One advantage of
the local case is that the $L^p$ spaces are nested.  Because the $L^p(\R^n)$ are not nested,
we are not able to obtain equality of the union of $L^p(\R^n)$, $p>1$
and $\M_{A_1}(\R^n)$.  Remarkably, even the much larger union over
weak-$L^p(\R^n)$ spaces is not equal to $\M_{A_1}(\R^n)$.

\begin{theorem} \label{thm:mainglobal} 
If $p_0$ is any exponent with $1<p_0<\infty$ then
$$\bigcup_{p>1}L^{p,\infty}(\R^n)\subsetneqq \bigcup_{w\in A_{p_0}}
L^{p_0}_w(\R^n)\subsetneqq \M_{A_1}(\R^n).$$
\end{theorem}

The proof uses the extrapolation theory of Rubio de Francia
\cite{Rub1,Rub2} (see also the book \cite{CMP}).  In addition, the
global version of \ref{thm:localApmajor} does not hold.
\begin{theorem} \label{thm:globalApmajor} If $p>1$ then
$$\M_{A_1}(\R^n)\subsetneqq \M_{A_p}(\R^n).$$
\end{theorem}

There are some positive results on $\R^n$.  Define
$\M_{F}(\R^n)$ to be the class of functions such that $Mf<\infty$
a.e. on $\R^n$ and $\M_{A_\infty^{F}}$ to be functions, $f$, such that
there exists $w\in A_\infty^F=A_\infty\cap \M_F$ with $|f|\leq w$.
Given $w\in A_\infty$, a simple way to create a weight in $A_\infty^F$
is to take a truncation: let $w_\lambda=\max(w,\lambda)$ for $\lambda
>0$.  Then $w_\lambda \in A_\infty\cap L^\infty\subset A_\infty^F$.
Moreover, we have the following characterizations of $\M_{A_1}(\R^n)$.

\begin{theorem} \label{thm:globalAinfty}There holds $\M_{A_1}(\R^n)=\M_{A^F_\infty}(\R^n)$.
\end{theorem}

Remarkably, $\M_{A_1}(\R^n)$ can be described in terms of functions whose
maximal operator is only finite almost everywhere.

\begin{theorem}\label{thm:globalchar} 
A function $f$, belongs to $\M_{A_1}(\R^n)$ if and only if there is an
$s>1$ such that $|f|^s\in \M_F(\R^n)$.
\end{theorem}

We have two descriptions of the union $L^p_w(\R^n)$ where
$p$ is fixed and $w$ ranges over all $A_p$ weights.

\begin{theorem} \label{thm:A1L1union}Suppose $1<p<\infty$, then 
$$\bigcup_{w\in A_p} L^p_w(\R^n)=\M_{A_1}(\R^n)\cap\Big(\bigcup_{w\in A_1} L^1_w(\R^n)\Big).$$
\end{theorem}
The class $\M_{A_1}(\R^n)$ can be thought of as a generalization of
$L^\infty(\R^n)$---i.e.\ functions here are majorized by constants,
which are $A_1$ weights---while $\bigcup_{w\in A_1} L^1_w(\R^n)$ is
generalization of $L^1(\R^n)$.  Considering the basic fact
$$L^1(\R^n)\cap L^\infty(\R^n)\subset \bigcap_{1< p<\infty}L^p(\R^n)$$
Theorem \ref{thm:A1L1union} (combined with Theorem
\ref{thm:mainglobal}) shows that if we enlarge $L^\infty(\R^n)$ to
$\M_{A_1}(\R^n)$ and $L^1(\R^n)$ to $\bigcup_{w\in A_1}L^1_w(\R^n)$
and intersect the two, then we pick up an even bigger class of functions---one that properly contains the union of all $L^p(\R^n)$, for $p>1$.
We also note the following corollary to Theorem \ref{thm:A1L1union}.

\begin{corollary} If $1<p,q<\infty$ then 
$$\bigcup_{w\in A_p}L^p_w(\R^n)=\bigcup_{w\in A_q}L^q_w(\R^n).$$
\end{corollary}

Notice that if $w\in A_p(\R^n)$ then the maximal function is bounded
on $L^p_w(\R^n)$ and its dual space $L^{p'}_\sigma(\R^n)$, where
$\sigma=w^{1-p'}$ since $\sigma\in A_{p'}$.  It turns out this
property characterizes the union of all such function spaces.  For the
statement of our final theorem we will need the notion of a Banach
function space, which we refer the reader to Section \ref{prelim} for
more precise definitions.  Given a Banach function space, $\X$, we
denote the associate space by $\X'$ and we also write $M\in \B(\X)$ if
$M$ is a bounded operator on $\X$.  We end our introductory results with the following theorem that says a function belongs to a function space $\X$ for which the Hardy-Little maximal function is bounded on $\X$ and $\X'$ if and only if $f\in L^p_w(\R^n)$ for some $p>1$ and $w\in A_p(\R^n)$.

\begin{theorem} \label{thm:BFS} Suppose $1<p<\infty$ then
$$\bigcup_{w\in A_p}L^p_w(\R^n)=\bigcup\{\X:M\in
  \B(\X)\cap\B(\X')\},$$ where the second union is over all Banach
  function spaces such that the Hardy-Littlewood maximal operator is
  bounded on $\X$ and $\X'$.
\end{theorem} 

We also mention the recent result of Chu \cite{C} who proved that 
$$\mathcal M_{A_1}(\R^n)=\bigcup\{\X:M\in
  \B(\X)\}.$$

The rest of this paper will be as follows.  In Section \ref{prelim} we
state preliminary results necessary for the rest of the paper.  In
Section \ref{classMpr} we make some remarks about the classes of
functions with $A_p$ majorants that hold for a general domain.  In
Section \ref{local} we develop the local theory, proving Theorems
\ref{thm:localmainr}, \ref{thm:localmaininfty},
\ref{thm:localApmajor}, \ref{thm:Hp>0} and \ref{thm:Hp>1}.  Section
\ref{global} is devoted to the global theory, in particular we prove
Theorems \ref{thm:mainglobal}, \ref{thm:globalApmajor},
\ref{thm:globalAinfty}, and \ref{thm:globalchar}.  We finish the manuscript with some open questions in Section \ref{questions}.

{\bf Acknowledgement.} We would like to thank Carlos P\'erez and Javier Duoandikoetxea for comments that helped improve the quality of this manuscript.  We would like to thank David Cruz-Uribe for Example \ref{ex:example1}. 

\section{Preliminaries} \label{prelim}
As mentioned in the introduction it is imperative that we separate the
local and global cases.  Hereafter, $\Omega$ will denote either $\R^n$
or a cube, $Q$, with sides parallel to the axes in $\R^n$.  Let us
begin with definition of Lebesgue spaces.  Given $p$, $1\leq
p\leq\infty$, we will use $p'$ to denote the dual exponent defined by
the equation $1/p+1/p'=1$.  For $0<p<\infty$, $L^p(\Omega)$ is the set
of measurable functions such that
$$\|f\|_{L^p(\Omega)}^p=\int_{\Omega}|f|^p\,dx<\infty.$$ Given a cube
$Q$ a weight on $Q$ will be a positive function in $L^1(Q)$.  A weight
on $\R^n$ will be a positive function in $L^1_{\text{loc}}(\R^n)$.
Given a weight, $w$, define $L^p_w(\Omega)$ to be functions normed by
$$\|f\|^p_{L^p_w(\Omega)}=\int_{\Omega}|f|^pw\,dx.$$

Let $M_\Omega$ be the Hardy-Littlewood maximal operator restricted to $\Omega$, i.e.,
$$M_\Omega f(x)=\sup_{\substack{Q\subset \Omega \\ x\in Q}}\frac{1}{|Q|}\int_Q |f|\,dy.$$
When $\Omega=\R^n$ we write $M_{\R^n}f=Mf$.  

We define $A_1(\Omega)$ to be the class of all weights on $\Omega$ such that $M_\Omega w(x)\leq Cw(x)$ a.e. $x\in \Omega.$
While $A_p(\Omega)$, for $p>1$, is the class of all weights on $\Omega$ such that
$$\sup_{Q\subset \Omega}\Big(\frac{1}{|Q|}\int_Q w\,dx\Big)\Big(\frac{1}{|Q|}\int_Q w^{1-p'}\,dx\Big)^{p-1}<\infty.$$
Given an $A_p$ weight $w$ we will refer to the weight $\sigma=w^{1-p'}$ as the dual weight.
For the endpoint, $p=\infty$, we will use the definition
$$A_\infty(\Omega)=\bigcup_{p\geq 1} A_p(\Omega).$$
There are several other definitions of $A_\infty$, e.g., weights satisfying a reverse Jensen inequality, a reverse H\"older inequality, or fairness condition with respect to Lebesgue measure \cite{Duo1,GrCl}.  

\begin{example} Let $x_0\in \Omega$, $1\leq p\leq\infty$, and $w_{x_0}(x)=|x-x_0|^\al$.  Then $w_{x_0}\in A_p(\Omega)$ if and only if $-n<\al<n(p-1)$.
\end{example}

The following theorem states some elementary properties of $A_p$ weights, most of which follow from the definition (see \cite[Proposition 7.2]{Duo1}).
\begin{theorem} \label{thm:properties} The following hold.
\begin{enumerate}[(i)]
\item $A_1\subset A_p\subset A_q\subset A_\infty$ if $1<p<q<\infty$.
\item For $1<p<\infty$, $w\in A_p$ if and only if $\sigma=w^{1-p'}\in A_{p'}$.
\item If $0<s\leq 1$ and $w\in A_p$ then $w^s\in A_p$.
\item If $u,v\in A_1$ then $uv^{1-p}\in A_p$.
\end{enumerate}
\end{theorem}
It is interesting to note that the converse of (iv) also holds, but
the proof is much more intricate.  This was shown by Jones in
\cite{Jon}.  We emphasize that we will not need this converse
statement, only the statement (iv).  

We will also need the following deeper property of $A_\infty$ weights
known as the reverse H\"older inequality.  See \cite{HPR} for a simple
proof with nice constants.

\begin{theorem} \label{thm:rh} If $w\in A_\infty(\Omega)$, then there exists $s>1$ such that for every cube $Q\subset \Omega$
$$\frac{1}{|Q|}\int_Q w^s\,dx\leq \Big( \frac{2}{|Q|}\int_Q w\,dx\Big)^s.$$
\end{theorem}

As a corollary to Theorem \ref{thm:rh} we have the following openness properties of $A_p$ classes.
\begin{theorem} \label{thm:open} Let $1\leq p\leq \infty$.  The following hold
\begin{enumerate}[(i)]
\item $A_p(\Omega)=\bigcup_{1\leq q<p} A_q(\Omega)$.
\item If $w\in A_p(\Omega)$ then $w^s\in A_p(\Omega)$ for some $s>1$.
\end{enumerate}
\end{theorem}

For the results on $\R^n$ we will need the notion of a Banach function space.  We refer the reader to book by Bennett and Sharpley \cite[Chapter 1]{BS} for an excellent reference on the subject.  A mapping $\rho$, defined on the set of non-negative $\R^n$-measurable functions and taking values in $[0,\infty]$, is said to be a Banach function norm if it satisfies the following properties:
\begin{enumerate}[(i)]
\item $\rho(f)=0 \Leftrightarrow f= 0$ a.e., $\rho(af)=a\rho(f)$ for $a>0$, $\rho(f+g)\leq \rho(f)+\rho(g)$;
\item if $0\leq f\leq g$ a.e., then $\rho(g)\leq \rho(f)$;
\item if $f_n\uparrow f$ a.e., then $\rho(f_n)\uparrow \rho(f)$;
\item if $B\subset \R^n$ is bounded then $\rho(\chi_B)<\infty$;
\item if $B\subset \R^n$ is bounded then $$\int_B f\,dx\leq C_B\rho(f)$$ 
for some constant $C_B$, $0<C_B<\infty$.
\end{enumerate}
We note that our definition of a Banach function space is slightly different than that found in \cite{BS}.  In particular, in the axioms (iv) and (v) we assume that the set $B$ is a bounded set, whereas it is sometimes assumed that $B$ merely satisfy $|B|<\infty$.  We do this so that the spaces $L^p_w(\R^n)$ with $w\in A_p$ satisfy items (iv) and (v).  (See also the discussion at the beginning of Chapter 1 on page 2 of \cite{BS}.)

Given Banach function norm $\rho$, $\X=\X(\R^n,\rho)$, is the collection of a measurable functions such that $\rho(|f|)<\infty$.  In this case we may equip $\X$ with the norm
$$\|f\|_{\X}=\rho(|f|).$$
Given a Banach function space we may define the associate space, $\X'$, as all measurable functions, $g$, such that $fg\in L^1(\R^n)$ for all $f\in \X$.  This space may be the normed by
\begin{equation}\label{associate} \|g\|_{\X'}=\sup\Big\{\int_{\R^n} |fg|\,dx: \|f\|_\X\leq 1\Big\}.\end{equation}
Equipped with this norm $\X'$ is also a Banach function space and 
$$\int_{\R^n}|fg|\,dx\leq \|f\|_{\X}\|g\|_{\X'}.$$
Typical examples of Banach function spaces are $L^p(\R^n)$, $1\leq p\leq \infty$ whose associate spaces are $L^{p'}(\R^n)$.  Other Banach spaces include weak type spaces $L^{p,\infty}(\R^n)$, the Lorentz space $L^{p,q}(\R^n)$, and Orlicz spaces $L^\Phi(\R^n)$ defined for a Young function $\Phi$ (see \cite{BS, CMP}).   When $w\in A_p(\R^n)$ and $1\leq p\leq \infty$ the spaces $L^p_w(\R^n)$ are also Banach function spaces with respect to Lebesgue measure.  To see this it suffices to check property (v): For $1<p<\infty$ if $B$ is bounded then $B\subset Q$ for some cube $Q$ so $\sigma(B)<\infty$, where  and
$$\int_B f\,dx= \int_B f w^{1/p}w^{-1/p}\,dx\leq \sigma(B)^{1/p'}\|f\|_{L^p_w}.$$
The space $L^\infty_w(\R^n)$ for $w\in A_\infty(\R^n)$ can be realized simply as $L^\infty(\R^n)$ since $w\in A_\infty$ implies $w>0$ a.e.  Finally, for $L^1_w(\R^n)$ and $w\in A_1(\R^n)$ note that
\begin{equation} \label{L1embed}\int_Bf\,dx=\int_Bfww^{-1}\,dx\leq (\inf_B w)^{-1}\|f\|_{L^1_w}.\end{equation}

The associate space of $L^p_w(\R^n)$ defined by the pairing in \eqref{associate} is given by $L^{p'}_\sigma(\R^n)$ where $\sigma=w^{1-p'}$, not $L^{p'}_w(\R^n)$.  We will be particularly interested in Banach functions spaces $\X$ for which
$$\|Mf\|_{\X}\leq C\|f\|_{\X}$$
in which case we write $M\in \B(\X)$.

We end this section with the classical result of Coifman and Rochberg \cite{CR} (see also \cite[Theorem 3.4, p. 158]{GCRdF}).  This result requires a definition.
\begin{definition} We say that a function belongs to $\M_{F}(\Omega)$ if 
$$M_\Omega f(x)<\infty, \quad \text{for a.e.} \ x\in \Omega.$$
\end{definition}
If $f$ belongs to a Banach function space for which $M\in \B(\X)$ then $f\in \M_F$.

\begin{theorem} \label{thm:coifroch} If $f\in \M_{F}(\Omega)$ and $0<\delta<1$ then $(M_\Omega f)^\delta \in A_1(\Omega)$.
\end{theorem}

We leave the reader with the following table of notation that will be used throughout the manuscript.

\begin{center}
\begin{tabular}{cl}
$\Omega$& Domain of interest, either $\R^n$ or a cube $Q\subset \R$;\\
$M_\Omega$& Hardy-Littlewood maximal operator restricted to $\Omega$;\\
$A_p(\Omega)$& class of $A_p$ weights on $\Omega$;\\
$\M_{A_p}^r(\Omega)$&  functions on $\Omega$ with $|f|^r$ majorized by an $A_p$ weight;\\
$\M_F(\Omega)$ & functions on $\Omega$ such that $M_\Omega f<\infty$ a.e.;\\
$A_p^F(\Omega)$ & $A_p(\Omega)\cap \M_F(\Omega)$;\\
$\M_{A_p^F}(\Omega)$ & functions majorized by $A_p^F(\Omega)$ weights.

\end{tabular}
\end{center}

\section{The classes $\M^r_{A_p}$} \label{classMpr}

Let us now define a general class of functions majorized by $A_p$
weights and establish some properties of such classes.  We remind the
reader that throughout a domain $\Omega$ will denote all of $\R^n$ or
a cube $Q$ in $\R^n$.

\begin{definition}\label{defn:Apmaj} Let $r$ and $p$ satisfy $0<r<\infty$ and $1\leq p\leq \infty$.  Define $\M_{A_p}^r(\Omega)$ to be the collection of all measurable functions on $\Omega$, $f$, such that there exists $w\in A_p(\Omega)$ with
$$|f(x)|^r\leq w(x) \qquad \text{for} \ a.e.\ x\in \Omega.$$
When $r=1$ we simply write $\M_{A_p}(\Omega)$.
\end{definition}

Theorem \ref{thm:open} implies the following general facts about the  $\M^r_{A_p}$ classes.

\begin{theorem} Suppose $r$ and $p$ satisfy $0<r<\infty$ and $1\leq p\leq \infty$.  Then, there holds
\begin{equation}\label{eqn:openr}\M^r_{A_p}(\Omega)=\bigcup_{s>r} \M^s_{A_p}(\Omega)\end{equation}
and if $p>1$,
\begin{equation}\label{eqn:openp}\M^r_{A_p}(\Omega)=\bigcup_{1\leq q<p} \M^r_{A_q}(\Omega).\end{equation}
\end{theorem}
\begin{proof} We first prove \eqref{eqn:openr}.  It is clear from (iii) of Theorem \ref{thm:properties} that $\bigcup_{r<s}\M_{A_p}^s(\Omega)\subset \M_{A_p}^r(\Omega)$.  On the other hand, if $f\in \M^r_{A_p}(\Omega)$ then $|f|^r\leq w\in A_p$.  By (ii) of Theorem there exists $t>1$ such that $w^t\in A_p(\Omega)$, but then taking $s=rt>r$ and $u=w^t$ we have $|f|^s\leq u\in A_p$, so $f\in \bigcup_{r<s} \M^s_{A_p}(\Omega)$.  The proof of equality \eqref{eqn:openp} follows directly from (i) of Theorem \ref{thm:properties} and (i) of Theorem \ref{thm:open}.
\end{proof}

Our next Theorem shows that for a function to have an $A_1$ majorant it is equivalent for its maximal function to have an $A_1$ majorant.
\begin{theorem} \label{thm:A1max} $f\in \M_{A_1}(\Omega)$ if and only if $M_\Omega f\in \M_{A_1}(\Omega).$ 
\end{theorem}
\begin{proof} If $f\in \M_{A_1}(\Omega)$ then we have $M_\Omega f\leq M_\Omega w\leq Cw$ since $w\in A_1(\Omega)$, which is to say $M_\Omega f \in \M_{A_1}(\Omega)$.  The converse statement follows from the fact that $|f|\leq M_{\Omega} f$.
\end{proof}

Using the exact same reasoning it is easy to prove that $f\in \M^r_{A_1}(\Omega)$ if and only if $M_\Omega(|f|^r)\in \M_{A_1}(\Omega)$.  However, we can do slightly better when $r\geq 1$.

\begin{theorem} \label{thm:r>1} If $r\geq 1$ then the following are equivalent.
\begin{enumerate}[(i)]
\item $f\in \M^r_{A_1}(\Omega)$.
\item $M_\Omega(|f|^r)\in \M_{A_1}(\Omega)$.
\item $M_\Omega f\in \M_{A_1}^r(\Omega)$.
\end{enumerate}

\end{theorem}
\begin{proof} The equivalence (i) $\Leftrightarrow$ (ii) follows from Theorem \ref{thm:A1max}.  We will prove (ii) $\Rightarrow$ (iii) and (iii) $\Rightarrow$ (i).  

Suppose that $w\in A_1(\Omega)$ and $M_\Omega(|f|^r)\leq w$.  Since $r\geq 1$ we have $(M_\Omega f)^r\leq M_\Omega(|f|^r)\leq w$ which is to say that $M_\Omega f\in \M_{A_1}^r$.  

On the other hand if $(M_\Omega f)^r\leq w\in A_1(\Omega)$, so $M_\Omega f<\infty$ a.e., and hence $f$ is locally integrable on $\Omega$.  By the Lebesgue differentiation theorem we have 
$$|f|^r\leq (M_\Omega f)^r \leq w.$$
\end{proof}

For the case $0<r<1$ we still have that $f\in \M_{A_1}^r(\Omega)$ if and only if $M_\Omega(|f|^r)\in \M_{A_1}(\Omega)$, however it is not true that this is equivalent to $(M_\Omega f)^r\in \M_{A_1}(\Omega)$.  Consider the simple example.

\begin{example} Let $f(x)=|x|^{-n}$ on $Q=[-1,1]^n$.  If $0<r<1$ then $f\in M_{A_1}^r(Q)$ but $M_Q f\equiv\infty$.
\end{example}

Of course if $0<r<1$ and $M_\Omega f<\infty$ a.e., then $(M_\Omega f)^r\in A_1(\Omega)$ (hence $M_\Omega f\in \M_{A_1}^r(\Omega)$) automatically by Theorem \ref{thm:coifroch}.

\section{The local case} \label{local}

For this section $Q$ will be a fixed cube in $\R^n$.  We begin
with a proof Theorem \ref{thm:localmainr}.

\begin{proof}[Proof of Theorem \ref{thm:localmainr}] We will prove the chain of containments:
$$\bigcup_{w\in A_{\frac{p_0}{r}}} L^{p_0}_w(Q)\subset \bigcup_{p>r}L^p(Q)\subset \M_{A_1}^r(Q)\subset\bigcup_{w\in A_{\frac{p_0}{r}}} L^{p_0}_w(Q).$$

\begin{itemize}

\item $\Big(\bigcup_{w\in A_{\frac{p_0}{r}}} L^{p_0}_w(Q)\subset \bigcup_{p>r}L^p(Q)\Big)$: Suppose $f\in L^{p_0}_w(Q)$ for some $w\in A_{\frac{p_0}{r}}(Q)$.  Set $q_0=p_0/r$, then by (ii) of Theorem \ref{thm:properties}, we have $\sigma=w^{1-q_0'}\in A_{q_0'}(Q)$ and by Theorem \ref{thm:rh} satisfies a reverse H\"older inequality:
$$ \Big(\dashint_{Q'} \sigma^s\,dx\Big)^{1/s}\leq C\dashint_{Q'}\sigma\,dx$$
for some $s>1$ and all $Q'\subseteq Q$.  This implies that $\sigma \in L^s(Q)$.  Define $\frac1q=\frac{1}{q_0}+\frac{1}{sq_0'}$ so that $q>1$ and let $p=rq>r$.  Then
\begin{align*}\Big(\int_{Q} |f|^p\,dx\Big)^{1/p}&=\Big(\int_{Q} |f|^{rq}w^{q/q_0}w^{-q/q_0}\,dx\Big)^{1/p}\\
&\leq \Big(\int_{Q}|f|^{p_0}w\,dx\Big)^{1/p_0}\Big(\int_{Q}\sigma^s\,dx\Big)^{1/(sq_0')}.
\end{align*}

\item $\Big(\bigcup_{p>r}L^p(Q)\subset \M_{A_1}^r(Q)\Big)$: If $f\in L^p(Q)$ for some $p>r$ then $|f|^r\leq M_Q(|f|^p)^{r/p}\in A_1(\Omega)$ by Theorem \ref{thm:coifroch}.

\item $\Big(\M_{A_1}^r(Q)\subset \bigcup_{w\in A_{\frac{p_0}{r}}} L^{p_0}_w(Q)\Big)$: Set $q_0=p_0/r>1$ and suppose $g=|f|^r\leq w\in A_1(Q)$.  Then $w^{1-q_0}\in A_{q_0}(Q)$ by (iv) of Theorem \ref{thm:properties} and
$$\int_Q |f|^{p_0}w^{1-q_0}\,dx=\int_Q g^{q_0}w^{1-q_0}\,dx\leq \int_Q w\,dx<\infty.$$
\end{itemize}

\end{proof}

Before we move on, we observe the following Corollary.
\begin{corollary} If $1<p,q<\infty$ then
$$\bigcup_{w\in A_p}L^p_w(Q)=\bigcup_{w\in A_q}L^q_w(Q).$$
\end{corollary}

\begin{proof}[Proof of theorem \ref{thm:localmaininfty}] We first prove
$$\bigcup_{r>0}\M^r_{A_1}(Q)=\bigcup_{p>0}L^p(Q).$$
\begin{itemize}
\item ($\subset$): If $f\in \M^r_{A_1}(Q)$ for some $r>0$ and $w\in A_1(Q)$ is such that $|f|^r\leq w,$ then   $f\in L^r(Q)\subset \bigcup_{p>0}L^p(Q)$. 

\item ($\supset$):  If $f\in L^p(Q)$ for some $p>0$ let $r$ be such that $0<r<p$, then $|f|^r\leq M_Q(|f|^p)^{r/p}\in A_1(Q)$.
\end{itemize}

Next we show
$$\bigcup_{p>0}L^p(Q)=\bigcup_{w\in A_\infty}L^{p_0}_w(Q).$$

\begin{itemize}

\item ($\subset$): Suppose $f\in L^p(Q)$ for some $0<p<\infty$, then if $r<\min(p,p_0)$ we have
$$f\in L^p(Q)\subset \bigcup_{r<p} L^p(Q)=\bigcup_{w\in A_{\frac{p_0}{r}}}L^{p_0}_w(Q)\subset \bigcup_{w\in A_{\infty}} L^{p_0}_w(Q).$$

\item ($\supset$): Suppose $f\in L^{p_0}_w(Q)$ for some $w\in A_{\infty}$, then $w\in A_q$ for some $q>1$.  Set $p=p_0/q$ and notice that $p<p_0$.  Then
$$\int_{Q}|f|^p\,dx=\int_{Q}|f|^p w^{1/q}w^{-1/q}\,dx\leq \Big(\int_{Q}|f|^{p_0}w\,dx\Big)^{1/q}\Big(\int_{Q}w^{1-q'}\,dx\Big)^{1/q'}.$$

\end{itemize}

\end{proof}

\begin{example} \label{ex:example1} The function 
\begin{equation}\label{example1}f(x)=x^{-1}(\log x)^{-2}\chi_{(0,1/2)}(x)\end{equation}
does not belong to $\M_{A_1}([0,1])$.  This follows from Theorem \ref{thm:localmain} since it can be readily checked that 
$$f\in L^1([0,1])\backslash\Big(\bigcup_{p>1}L^p([0,1])\Big).$$
However, $f\in \M_{F}([0,1])$ since $f\in L^1([0,1])$. 
\end{example}

Before we prove Theorem \ref{thm:localApmajor} we make an observation: In order to get a smaller class of functions than $L^p(Q)$ one has to union over $w\in A_p$ for $p>1$. 

\begin{remark} \label{LpA1union} Suppose $0<p<\infty$, then 
$$L^p(Q)=\bigcup_{w\in A_1} L^p_w(Q).$$
\end{remark}

The proof of the equality in Remark \ref{LpA1union} follows the fact that $1\in A_1$ and inequality\eqref{L1embed} with $B=Q$.

\begin{proof}[Proof of Theorem \ref{thm:localApmajor}] 
It suffices to show $\M_{A_\infty}(Q)\subset \M_{A_1}(Q)$.  Suppose
that $f\in \M_{A_\infty}(Q)$ so that there exists $w\in A_\infty(Q)$
with
$$|f|\leq w.$$ Since $w\in A_\infty(Q)$, the reverse H\"older
inequality implies that there exists $s>1$ such that
$$(M_Qw^s)^{1/s}\leq 2M_Qw\leq 2(M_Qw^s)^{1/s}.$$ Moreover, since
$w\in L^1(Q)$, $M_Qw<\infty$ a.e., and hence by Theorem
\ref{thm:coifroch} $M_Qw\in A_1(Q)$ and hence $f\in \M_{A_1}(Q)$.
\end{proof}

\begin{proof}[Proof of Theorems \ref{thm:Hp>0} and \ref{thm:Hp>1}]
Since $N^{+}\cap L^{p}(\T) = H^p$ for $p>0$ (\cite{Dur} Theorem 2.11), we see that
$$H_{A_1}(\T)=N^+\cap \M_{\A_1}(\T)   = N^{+} \cap \bigcup_{p>1}L^p(\T) =
\bigcup_{p>1} H^p.$$ This is the first part of Theorem \ref{thm:Hp>1}.

To go from equality of the analogous $L^p$ spaces to the Hardy spaces
is a matter of using two facts for $0<p_0<\infty$
\begin{enumerate}
\item
\[
\int_\T \log w \,d\theta> -\infty \text{ and } w \in L^1(\T) \text{ implies }
w=|h|^{p_0}
\]
for some outer function $h \in H^{p_0}$.

\item If $h \in H^{p_0}$ is outer, then the set $h\C[z] = \vee\{z^jh: j \geq
  0\}$ is dense in $H^{p_0}$.
\end{enumerate}
Item (1) comes from the standard construction of an outer function
(see \cite{Dur} Section 2.5).  As for item (2), when $1\leq p_0
<\infty$ this is a standard generalization of Beurling's theorem
(\cite{Dur} Theorem 7.4).  When $0<p_0 < 1$, this is a less well known
result that can be found in Gamelin \cite{Gam}, Theorem 4.

For Theorem \ref{thm:Hp>1} we must show for $1<p_0<\infty$
\[
\bigcup_{p>1} H^p = \bigcup_{w \in A_{p_0}} H^{p_0}_w.
\]

Now, for $f \in H^p\subset L^p$, we know there exists $w \in
A_{p_0}(\T)$ such that $f \in L^{p_0}_w(\T)$ by Theorem
\ref{thm:localmain}.  Factor $w=|h|^{p_0}$ with outer $h\in
H^{p_0}$.  Then, $fh \in N^{+} \cap L^{p_0}(\T) = H^{p_0}$ while $h
\C[z]$ is dense in $H^{p_0}$ so that there exist polynomials $Q_n$
satisfying
\[
\int |fh-Q_n h|^{p_0}\,d\theta = \int |f-Q_n|^{p_0} w\,d\theta \to 0
\]
as $n\to \infty$. This shows $f \in H^{p_0}_w$ (since it is initially
defined as the closure of the analytic polynomials in
$L^{p_0}_w(\T)$).

Conversely, we have seen that if $f \in H^{p_0}_w$, then $f \in L^p(\T)$
for some $p>1$.  Factor $w=|h|^{p_0}$ as before. Then, $fh \in
H^{p_0}$ and $1/h$ is outer, so that $f=fh(1/h) \in N^{+}$.  Since $f
\in L^p(\T)$, we can then conclude that $f \in H^p$.  

The proof of Theorem \ref{thm:Hp>0}, which claims for $0<p_0<\infty$
\[
\bigcup_{p>0} H^p=\bigcup_{w\in A_\infty} H^{p_0}_w,
\]
is similar once we know the corresponding fact for $L^p(\T)$
spaces. Indeed, take $f \in H^p$ for some $p>0$. There exists $w \in
A_{\infty}$ such that $f \in L^{p_0}_{w}(\T)$ by Theorem
\ref{thm:localmaininfty}.  Factor $w = |h|^{p_0}$ with outer $h \in
H^{p_0}$.  Then, $f \in H^{p_0}_w$ as above using Gamelin's result.
The converse is similar to the previous proof.  
\end{proof}

\section{The global case} \label{global}
In this section we address the case when our functions are defined on all of $\R^n$.  Let us first prove Theorem \ref{thm:A1L1union}, which states that
$$\bigcup_{w\in A_p}L^p_w(\R^n)=\M_{A_1}(\R^n)\cap \bigcup_{w\in A_1}L^1_w(\R^n).$$

 \begin{proof}[Proof of Theorem \ref{thm:A1L1union}] We will prove the containment each direction

 First we show $$\M_{A_1}(\R^n)\cap \bigcup_{w\in A_1}L^1_w(\R^n)\subset \bigcup_{w\in A_{p}}L^{p}_w(\R^n).$$  
 Suppose $w$ is an $A_1$ majorant of $f$ and $f\in L^1_u(\R^n)$ for some $u\in A_1(\R^n)$.   By Theorem \ref{thm:properties} $uw^{1-p}\in A_{p}(\R^n)$ and
 $$\int_{\R^n}|f|^{p}w^{1-p}u\,dx\leq \int_{\R^n}|f|u\,dx.$$
To see the reverse containment suppose that $f\not\equiv 0$, belongs to $L^{p}_w(\R^n)$ for some $w\in A_{p}(\R^n)$.  We will use the fact that $w\in A_p(\R^n)$ implies $M\in \B(L^p_w)$ to apply the Rubio de Francia algorithm: 
$$Rf=\sum_{k=0}^\infty \frac{M^k f}{2^k\|M\|^k_{\B(L^{p}_w)}}.$$
Then $Rf$ is an $A_1$ majorant of $f$ so $f\in \M_{A_1}(\R^n)$.  Also let $g$ be any function in $L^{p'}_\sigma(\R^n)$ where $\sigma=w^{1-p'}$ satisfying $\|g\|_{L^{p'}_\sigma(\R^n)}~=~1$.  Again, since $\sigma\in A_{p'}(\R^n)$, we apply the Rubio de Francia algorithm
$$Rg=\sum_{k=0}^\infty \frac{M^kg}{2^k\|M\|_{\B(L^{p'}_\sigma)}^k},$$
so that $Rg$ is in $A_1(\R^n)$ and $\|Rg\|_{L^{p'}_\sigma(\R^n)}\leq 2$.  Hence
\begin{multline*}\int_{\R^n}|f|Rg\,dx=\int_{\R^n}|f|w^{1/p}Rg w^{-1/p}\,dx \\
\leq \|f\|_{L^{p}_w(\R^n)}\|Rg\|_{L^{p'}_\sigma(\R^n)}\leq 2  \|f\|_{L^{p}_w(\R^n)},\end{multline*}
showing that $f\in \bigcup_{w\in A_1}L^1_w(\R^n)$ as well.

\end{proof}

Before we move on, we remark that the intersection of $\M_{A_1}(\R^n)$ and $\bigcup_{w\in A_1}L^1_w(\R^n)$ is necessary for our result.  The function in example \ref{ex:example1} viewed as a function on $\R$ belongs to $L^1(\R)\subset \bigcup_{w\in A_1}L^1_w(\R)$, but does not belong to $L^p_w(\R)$ for any $p>1$ and $w\in A_p(\R)$ since it belongs to $L^1_{\text{loc}}(\R)\backslash\bigcup_{p>1}L^p_{\text{loc}}(\R)$.  We did not encounter this phenomenon in the local case since for a fixed cube $Q$, $\M_{A_1}(Q)\subset L^1(Q)$.

We now prove Theorem \ref{thm:BFS}.

\begin{proof}[Proof of Theorem \ref{thm:BFS}]
By Theorem \ref{thm:mainglobal}, it suffices to show 
\begin{equation}\label{eqn:BFSconAp} \bigcup_{w\in A_p} L^p_w(\R^n)\subset \bigcup\{\X:M\in \B(\X)\cap \B(\X')\}.\end{equation}
and
\begin{equation} \label{eqn:A1L1} \bigcup\{\X:M\in \B(\X)\cap \B(\X')\}\subset \M_{A_1}(\R^n)\cap \bigcup_{w\in A_1}L^1_w(\R^n).\end{equation}

 However, the containment \eqref{eqn:BFSconAp} is immediate, since 
 $$ M\in \B(L^p_w(\R^n))\Leftrightarrow w\in A_p(\R^n)\Leftrightarrow \sigma\in A_{p'}(\R^n) \Leftrightarrow M\in \B(L^{p'}_\sigma(\R^n)).$$ 
On the other hand, for containment \eqref{eqn:A1L1}, if $f\not\equiv 0$, $f\in \X$,  for some Banach function space $\X$ such that $M\in \B(\X)\cap \B(\X')$.  Then we may use the Rubio de Francia algorithm to construct an $A_1(\R^n)$ majorant:
$$Rf=\sum_{k=0}^\infty \frac{M^kf}{2^k\|M\|_{\B(\X)}^k}.$$
Then $Rf\in A_1$ and $|f|\leq Rf$ so $f\in \M_{A_1}(\R^n)$.  Given $g\in \X'$ let
$$Rg=\sum_{k=0}^\infty \frac{M^kg}{2^k\|M\|_{\B(\X')}^k},$$
so that $Rg\in A_1(\R^n)\cap \X'$ and $\|Rg\|_{\X'}\leq 2\|g\|_{\X'}$.  Then
$$\int_{\R^n}|f|Rg\,dx\leq \|f\|_\X\|Rg\|_{\X'}\leq 2\|f\|_{\X}\|g\|_{\X'}$$
so $f\in \bigcup_{w\in A_1}L^1_w(\R^n)$.
\end{proof}
When $p>1$, $L^{p,\infty}(\R^n)$ is a Banach function space such that $M$ is bounded on $L^{p,\infty}(\R^n)$ (see \cite{GrCl}) and its associate $(L^{p,\infty}(\R^n))'=L^{p',1}(\R^n)$, the Lorentz space with exponents $p'$ and $1$, is also a space for which $M$ is bounded (see \cite{AM}).

\begin{corollary}\label{cor:weak} Suppose $1<p<\infty$ then 
$$\bigcup_{p>1}L^{p,\infty}(\R^n)\subset \bigcup_{w\in A_p}L^p_w(\R^n).$$
\end{corollary}

From Corollary \ref{cor:weak} we see that the analogous version of Theorem \ref{thm:localmain} is not true on $\R^n$.  This follows since
$$\bigcup_{p>1}L^p(\R^n)\subsetneqq \bigcup_{p>1}L^{p,\infty}(\R^n)$$
for example $f(x)=|x|^{-n/2}\in L^{2,\infty}(\R^n)$ but $f\notin \bigcup_{p>0}L^p(\R^n)$.

We also remark that the techniques required for $\R^n$ are completely different than the local case.  For example, to prove the containment
$$\bigcup_{p>1}L^{p,\infty}(\R^n)\subset \M_{A_1}(\R^n)$$
it is not enough to simply dominate $|f|$ by $M(|f|^p)^{1/p}$.  Indeed for $f\in L^{p,\infty}(\R^n)$ we have $g=|f|^p\in L^{1,\infty}(\R^n)$, but $Mg$ may not be finite for $g\in L^{1,\infty}(\R^n)$ (take $g(x)=|x|^{-n}$ for example).  Instead we must refine our construction of an $A_1$ majorant using the techniques of Rubio de Francia \cite{Rub1}.

We now provide examples to show that the inclusions in Theorem \ref{thm:mainglobal} are proper.  We first show that the second inclusion is proper, i.e.,
$$\bigcup_{w\in A_p} L^p_w(\R^n)\subsetneqq \M_{A_1}(\R^n).$$
Since 
$$\bigcup_{w\in A_p} L^p_w(\R^n)=\M_{A_1}(\R^n)\cap \bigcup_{w\in A_1}L^1_w(\R^n),$$
it suffices to find a function in $\M_{A_1}(\R^n)\backslash \Big( \bigcup_{w\in A_1}L^1_w(\R^n)\Big)$.
\begin{example} The function $f(x)=1$ belongs to $\M_{A_1}(\R^n)\backslash \bigcup_{w\in A_1} L^1_w(\R^n)$.  To prove this we need the following fact: if $w\in A_\infty$ then $w\notin L^1(\R^n)$.  Suppose not, that is, suppose $w\in A_\infty(\R^n)\cap L^1(\R^n)$.  By Theorem \ref{thm:rh}, there exists $s>1$ such that
$$\Big(\frac{1}{|Q|}\int_Q w^s\,dx\Big)^{1/s}\leq \frac{2}{|Q|}\int_{Q} w\,dx.$$
Let $Q_N=[-N,N]^n$, then
\begin{multline*}\Big(\frac{1}{|Q_N|}\int_{Q_1}w^s\,dx\Big)^{1/s} \leq \Big(\frac{1}{|Q_N|}\int_{Q_N}w^s\,dx\Big)^{1/s}\\ \leq \frac{2}{|Q_N|}\int_{Q_N} w\,dx\leq \frac{2}{|Q_N|}\|w\|_{L^1(\R^n)}.\end{multline*}
Letting $N\ra \infty$ we arrive at a contradiction.  Finally to see $1~\notin~\bigcup_{w\in A_1}~L^1_w(\R^n)$, notice that $1\in L^1_w(\R^n)$ if and only if $w\in L^1(\R^n)$.

\end{example}
 
 Next we show that 
 $$\bigcup_{p>1}L^{p,\infty}(\R^n)\subsetneqq \bigcup_{w\in A_p}L^p_w(\R^n).$$
\noindent For this example we need the following lemma.
\begin{lemma} \label{lem:minmaxA1}Suppose $u,v\in A_1(\R^n)$ then
$$\max(u,v)\in A_1(\R^n) \qquad \text{and} \qquad \min(u,v)\in A_1(\R^n).$$
\end{lemma}
\begin{proof} 
To see that $\max(u,v)$ is in $A_1(\R^n)$ note that $\max(u,v)\leq u+v\leq 2\max(u,v)$ hence
$$M(\max(u,v))\leq Mu+Mv\leq C(u+v)\leq 2C\max(u,v).$$
To prove $\min(u,v)\in A_1(\R^n)$ we use the equivalent definition of $A_1(\R^n)$:
$$w\in A_1(\R^n) \Leftrightarrow \dashint_Q w\,dx\leq C\inf_Q w\quad
\forall Q\subset\R^n$$ where the infimum is the essential infimum of $w$ over the cube $Q$.
Set $w=\min(u,v)$ and let $Q$ be a cube.  Notice that $\inf_Q u>\inf_Q
v$ implies $\inf_Qw=\inf_Qv$ and hence
$$\dashint_Q w\,dx\leq \dashint_Q v\leq C\inf_Q v=C\inf_Q w.$$
On the other hand if $\inf_Q u\leq\inf_Q v$ then $\inf_Qw=\inf_Q u$ and so
$$\dashint_Qw\,dx\leq \dashint_Q u\,dx\leq C\inf_Q u=C\inf_Qw.$$
So $w\in A_1(\R^n)$.
\end{proof}
\begin{example}
 Let $f(x)=\max(|x|^{-\al n },|x|^{-\beta n})$.  If  $0<\al<\beta<1$ then $f\notin \bigcup_{p>0}L^{p,\infty}(\R^n)$.  However, 
$$|f(x)|\leq w(x)$$
where $w(x)=\max(|x|^{-\beta n},1)$ and $f\in L^1_u(\R^n)$ where $u(x)=\min(|x|^{-\gamma}, 1)$ when $1-\al<\gamma<1$.  By Lemma \ref{lem:minmaxA1} $u$ and $w$ belong to $A_1(\R^n)$.  Thus 
$$f\in \M_{A_1}(\R^n)\cap \bigcup_{w\in A_1}L^1_w(\R^n)=\bigcup_{w\in A_p}L^p_w(\R^n).$$
\end{example}

\begin{example}[Proof of Theorem \ref{thm:globalApmajor}]  Let $p>1$ and $0<\al<n(p-1)$.  Now consider the function $f(x)=|x|^\al$.  Then $f\in A_p(\R^n)\subset \M_{A_p}(\R^n)$, but $f\notin \M_F(\R^n)$ so in particular, $f\notin \M_{A_1}(\R^n)$.  To see this notice that for every $x\in \R^n$, and $r>|x|$
$$Mf(x)\geq \frac{c}{r^n}\int_{|x|\leq r}|x|^\al\,dx \simeq r^\al$$
so $Mf\equiv \infty$.
\end{example}

 \begin{proof}[Proof of Theorem \ref{thm:globalAinfty}]  Since $A_1(\R^n)\subset A_\infty(\R^n)$ and $A_1(\R^n)\subset \M_F(\R^n)$ we have $\M_{A_1}(\R^n)\subset \M_{A_\infty^F}(\R^n)$.  On the other hand if $f$ is dominated by a weight $w\in A_\infty^F(\R^n)=A_\infty(\R^n)\cap \M_F(\R^n).$  Then, by Theorem \ref{thm:rh} we have
 $$M(w^s)^{1/s}\leq 2Mw<\infty \qquad a.e.$$
for some $s>1$. So in particular $|f|\leq M(|f|^s)^{1/s}\leq M(w^s)^{1/s}\in A_1(\R^n)$.
 
 \end{proof} 
 
\begin{proof}[Proof of Theorem \ref{thm:globalchar}] Let $w$ be an $A_1(\R^n)$ majorant of $f$.  Since $w\in A_1(\R^n)$, $w^s\in A_1(\R^n)$ for some $s>1$, which implies $|f|^s\in \M_{A_1}(\R^n)$.  By Theorem \ref{thm:r>1} we have $M(|f|^s)\in \M_{A_1}(\R^n)\subset L^1_{\loc}(\R^n)$.  On the other hand if there exists $s>1$ such that $M(|f|^s)<\infty$ a.e. then $M(|f|^s)^{1/s}\in A_1(\R^n)$ by Theorem \ref{thm:coifroch} and $|f|\leq M(|f|^s)^{1/s}.$
\end{proof}

Finally we end with a brief description of $\bigcup_{w\in A_1}L^1_w(\R^n)$.  
\begin{theorem} 
$$\bigcup_{w\in A_1}L^1_w(\R^n)=\bigcup\{\X: M\in \B(\X')\}=\bigcup\{\X: \X'\cap A_1(\R^n)\not=\varnothing\}$$
\end{theorem}
\begin{proof}  It is clear that that
$$\bigcup_{w\in A_1}L^1_w(\R^n)\subset\bigcup\{\X: M\in \B(\X')\},$$
since the dual space of $L^1_w(\R^n)$ is $L^\infty(\R^n)$ and $M\in \B(L^\infty)$.  The associate space is always a closed subspace of the dual space \cite{BS,Rub2}.  Suppose $f\in \X$ such that $M\in \B(\X')$ then given $g\in \X'$ with $g\not\equiv 0$ (notice Banach function spaces always contain non-zero functions by (iv)) we may define 
$$w=\sum_{k=1}^\infty \frac{M^k g}{2^k\|M\|_{\B(\X')}^k}$$
so that $w\in A_1(\R^n)$ and $\|w\|_{\X'}\leq \|g\|_{\X'}$, so $w\in  \X'\cap A_1(\R^n)$.  Finally suppose $f\in \X$ for some $\X$ such that $\X'$ contains an $A_1$ weight.  Let $w\in \X'\cap A_1(\R^n)$.  Then
$$\int_{\R^n} |f|w\,dx\leq \|f\|_{\X}\|w\|_{\X'}$$
\end{proof}

\section{Questions} \label{questions} 
We leave the reader with some open questions.

\begin{enumerate}[{\bf 1.}]
\item Let $A_p^*=\bigcap_{q>p}A_q$.  Is there a characterization of the union
$$\bigcup_{w\in A_p^*}L^p_w?$$
 In general $A_p\subsetneqq A_p^*$.  For example $w(x)=\max\big((\log|x|^{-1})^{-1},1\big)$ belongs to $A_1^*$ but not $A_1$.  Moreover, 
$$\{w:w,1/w\in A^*_1\}=\text{clos}_{BMO}L^\infty,$$
see \cite{GCRdF, JN}.  In the local case we have
$$\bigcup_{w\in A_p^*}L^p_w(Q)\subset \bigcap_{s<p}\bigcup_{r>s}L^r(Q)=\limsup_{r\ra p^-} L^r(Q).$$
Are these two sets equal?
\item Can one give a better description of $\M_{A_1}(\R^n)$?  Using the techniques of the paper it is easy to show that
$$\bigcup\{\X:M\in \B(\X)\}\subset \M_{A_1}(\R^n).$$
Are these two sets equal?  Since first posting of this article, this
question has been answered positively by Chu \cite{C}.
\item It is well known that 
$$L^1\cap L^\infty \subset \bigcap_{1<p<\infty} L^p\subset \bigcup_{1<p<\infty} L^p \subset L^1+L^\infty.$$
When can we write a function as the sum of a function in $\M_{A_1}$ and $\bigcup_{w\in A_1}L^1_w$, that is, what conditions on a function guarantee it belongs to $\M_{A_1}+\bigcup_{w\in A_1}L^1_w$?
\item What can one say about 
$$\bigcup_{w\in A_p} L^{p,\infty}_w?$$
If $w\in A_1$ and $p>1$ then $M\in \B(L^{p,\infty}_w)$, so for $p>1$
$$\bigcup_{w\in A_1} L^{p,\infty}_w\subset \M_{A_1}.$$
\item Do these results transfer to more general domains?  It is possible to consider a general open set $\Omega$ as our domain of interest.  We may define the $A_p(\Omega)$ classes, $\M_{A_1}(\Omega)$, and the Hardy-Littlewood maximal operator $M_\Omega$ exactly as before.  However, the openness results, Theorems \ref{thm:rh} and \ref{thm:open}, may not hold for $\Omega$, even if it is bounded \cite{CMP}.   In the local case we assume that weights belong to $L^1(\Omega)$.  What happens if we only assume $L^1_{\text{loc}}(\Omega)$?

\end{enumerate}

\bibliographystyle{plain}

\end{document}